\documentclass[12pt,twoside,reqno]{amsart}

\allowdisplaybreaks
\newtheorem{thm}{Theorem}[section]
\newtheorem{lemma}{Lemma}[section]
\newtheorem{rem}{Remark}[section]
\newtheorem{prop}{Proposition}[section]

\theoremstyle{definition}

\theoremstyle{remark}

\newcommand{\R}{{\mathbb R}}
\newcommand{\N}{{\mathbb N}}

\newcommand{\noi}{\noindent}

\numberwithin{equation}{section}

\begin{document}
\title[Benney-Lin]	
{Global solutions for the generalized Benney-Lin equation  posed on bounded intervals and on a half-line.
}
\author{N. A. Larkin}

\address
{
	Departamento de Matem\'atica, Universidade Estadual
	de Maring\'a, Av. Colombo 5790: Ag\^encia UEM, 87020-900, Maring\'a, PR, Brazil}
	\email{ nlarkine@uem.br;nlarkine@yahoo.com.br}

\date{}

\begin{abstract}Initial boundary value problems for the generalized Benney-Lin equation posed on bounded intervals  and on the right half-line were considered. The existence and uniqueness of  global  regular solutions  on arbitrary intervals as well as their exponential decay for small solutions and for a special choice of a bounded interval have been established.
\end{abstract}

\maketitle

\section{Introduction}\label{introduction}
This work concerns the existence and uniqueness of global  solutions, regularity  and  exponential decay rates of solutions for some initial boundary value problems of the generalized Benney-Lin equation
\begin{align}
	& u_t+\eta D^5_x u+\beta D^4_x u+\alpha D^3_x u+ \gamma D^2_x u +u^ku_x=0,
\end{align} where $k\geq 1$ is a natural number.
This equation was deduced in \cite{benney} in connection with applications in fluid mechanics and later was exploited in \cite{lin} during research in the theory of liquid films. For various values of the coefficients 
$\alpha,\eta,\beta,\gamma$ and $k=1$, it presents well known equations of mathematical physics such as the Korteweg-de Vries equation when $\eta=\beta=\gamma=0$ and $\alpha=1$;  the Kawahara equation when $\gamma=\beta=\alpha=0$ and $\eta=-1$.
In the case $\eta=\alpha=0$ and $\beta, \gamma$ are positive constants, (1.1) is the Kuramoto-Sivashinsky equation 
In \cite{kuramoto}, Kuramoto studied the turbulent phase waves  and Sivashinsky in \cite{sivash} obtained an asymptotic equation which simulated the evolution of a disturbed plane flame front. See also \cite{Cross,cuerno}.
Mathematical  results on  initial and initial boundary value problems for various variants of (1.1)  are presented in \cite{araruna,Iorio, Guo,cousin,dor,familark,feng,lar,Larkin,larluch1,larluch2,marcio1,marcio2,temam1,saut2,temam2,zhang}, see references  there for more information. In  \cite{Iorio, feng,Larkin, zhang}, Kuramoto-Sivashinsky type equations have been considered which included $u_{xxx}$ (KdV) term. 
What concerns (1.1) with $k>1, \;\eta=\beta=\gamma=0;\;\alpha=1$, called generalized Korteweg-de Vries equations, the Cauchy problem  for (1.1) has been studied in \cite{farah,Fonseka1,Fonseka2,martel,merle}, where it has been proved that for $k=4$, called the critical case,  the initial problem is well-posed for small  initial data, whereas for arbitrary initial data, solutions may blow-up in a finite time. The generalized KdV equation was intensively studied in order to understand the interaction between the dispersive term and  nonlinearity in the context of the theory of nonlinear dispersive evolution equations \cite{jeffrey,kaku}. In \cite{lipaz}, an initial-boundary value problem for the generalized KdV equation with an internal damping posed on a bounded interval was studied in the critical case and in \cite{lar4}, an initial boundary value problem for the generalized KdV equation posed on a half-line was considered. Exponential decay of weak solutions for small initial data has been established. In \cite{araruna}, decay of weak solutions for $\eta=-1,\;\alpha=\beta=\gamma=0;\;k=2$ has been studied.\\
Here we study some initial boundary value problems for the generalized Benney-Lin (Bl) equation
\begin{align}
& u_t- D^5_x u+ D^4_x u+ D^3_x u+ D^2_x u +u^ku_x=0,
\end{align} where $k\geq 1$ is a natural number.

First essential problem that arises while one studies stability of (1.2) is a destabilizing effect of 
$D^2_x u$ that may be damped by a dissipative term $D^4_x u$ provided an interval $(0,L)$ has some specific properties. 
Our work has the following structure: Chapter I is Introduction. Chapter 2 contains notations and auxiliary facts. In Chapter 3, formulation and study of an initial boundary value problem for (1.2) posed on an arbitrary finite interval and on the right half-line are given. The existence of a global regular  solution for $k\leq 4$ and any positive time interval $(0,T)$ was established. in Chapter 4, an initial boundary value problem for (1.2) posed on a special interval $(0,L)$ was formulated. The existence of a global regular solution, uniqueness and exponential decay rate of small solutions for $k\leq 8$ have been established. Moreover,  the "smoothing" effect has been observed. Chapter 5 contains conclusions.

\section{Notations and Auxiliary Facts}

Let $L$ be a positive number and $x \in (0,L).$ We use the standard notations of Sobolev spaces $W^{k,p}$, $L^p$ and $H^k$ for functions and the following notations for the norms \cite{Adams, Brezis}:
for scalar functions $f(x,t):$\\
$$\R^+=\{t\in\R^1;\;t\geq 0\},\;\;\| f \|^2 =(f,f)= \int_0^L | f |^2dx, \;\; \| f \|_{L^p}^p = \int_0^L | f  |^p\, dx,$$
$$\| f \|_{W^{k,p}}^p = \sum_{0 \leq \alpha \leq k} \|D^\alpha_x f \|_{L^p},\; D^\alpha_x=\frac {d^{\alpha}}{dx^{\alpha}},\;D_x=D^1_x,\;D^0_xu=u.$$
We will use also the standad notations: $$D_xf=f_x,\;D_x^2 f=f_{xx},\;\frac{\partial}{\partial t}f=f_t.$$
When $p = 2$, $W^{k,p} = H^k$ is a Hilbert space with the scalar product 
$$((u,v))_{H^k}=\sum_{|j|\leq k}(D^ju,D^jv),\;\|u\|_{\infty}=
\|u\|_{L^{\infty}(0,L)}=ess \sup_{(0,L)}|u(x)|.$$
We use a notation $H_0^k(0,L)$ to represent the closure of $C_0^\infty(0,L)$, the set of all $C^\infty$ functions with compact support in $(0,L)$, with respect to the norm of $H^k$.

\begin{lemma}[Steklov's Inequality \cite{steklov}] Let $v \in H^1_0(0,L).$ Then
	\begin{equation}\label{Estek}
	\frac {\pi^2}{L^2}\|v\|^2 \leq \|v_x\|^2.
	\end{equation}
\end{lemma}

\begin{lemma}(See \cite{larluch1}, Lemma 2.2.) 
	Let $u$ be either $ u\in H^2_0(0,L)$ or $u\in H^2(\R^+),\;u(0)=D_x(0)=0.$ Then the following inequality holds:
	\begin{equation}\label{GN1}
		\|u\|_{\infty}\leq \sqrt{2}\|D^2u\|^{\frac{1}{4}}\|u\|^{\frac{3}{4}}.
	\end{equation}
\end{lemma}

\begin{lemma}\label{*}(See \cite{niren}, p. 125).
	Suppose $u$ and $D^mu$, $m\in\mathbb{N},$ belong to $L^2(0,L)$. Then for the derivatives $D^iu$, $0\leq i<m$, the following inequality holds: 
	\begin{equation}\label{GN2}
		\|D^iu\|\leq A_1\|D^mu\|^{\frac{i}{m}}\|u\|^{1-\frac{i}{m}}+A_2\|u\|,
	\end{equation}
	where $A_1$, $A_2$ are constants depending only on $L$, $m$, $i$.
\end{lemma}
\begin{lemma}
	[Differential form of the Gronwall Inequality]\label{gronwall} Let $I = [t_0,t_1]$. Suppose that functions $a,b:I\to \R$ are integrable and a function $a(t)$ may be of any sign. Let $u:I\to \R$ be a differentiable function satisfying
	\begin{equation}
		u_t (t) \leq a(t) u(t) + b(t),\text{ for }t \in I\text{ and } \,\, u(t_0) = u_0,
	\end{equation}
	then
	$$u(t) \leq u_0 e^{ \int_{t_0}^t a(\tau)\, d\tau } + \int^t_{t_0} e^{\int_{t_0}^s a(r) \, dr} b(s) ds$$\end{lemma}

\section{ Generalized Benney-Lin equation posed on  bounded intervals. Problem I}

Define an interval  $$D=x\in   (0,L>0);\;Q_t=(0,t)\times D.$$
 \begin{lemma}
	Let  $f\in H^2(D)\cap H^1_0(D).$  Then
	\begin{align}
		&a\|f\|^2\leq\| f_x\|^2,\;\;a^2\|f\|^2\leq \| f_{xx}\|^2,\;\;a\|f_x\|^2\leq \| f_{xx}\|^2,\\
		&\text{where} \; a=\frac{\pi^2}{L^2}.
	\end{align}	
\end{lemma}

\begin{proof} 
	Making use of Steklov`s inequalities, we get
	$$\| f_x\|^2\geq \frac{\pi^2}{L^2}\|f\|^2=a\|f\|^2.$$
	On the other hand,
	$$a\|f\|^2\leq \| f_x\|^2 =-\int_0^Lf f_{xx}dx\leq \|f_{xx}\|\|f\|.$$
	This implies
	$$a\|f\|\leq \|f_{xx}\|\;\;\text{and} \;\;a^2\|f\|^2\leq \| f_{xx}\|^2.$$
	Consequently, \;$a\| f_x\|^2\leq \| f_{xx}\|^2.$ \\
	Proof of Lemma 3.1 is complete.
\end{proof}
In $Q_{t}$ consider the following initial boundary value problem:

\begin{align}
 u_t- D^5_x u+ D^4_x u+ D^3_x u+ D^2_x u +u^ku_x=0,&\\
 D^i_x(0)= D^i_x(L)=D_x^2(L)=0,\;i=0,1;&\\
 u(x,0)=u_0(x).
\end{align}
\begin{thm} Let $T, L$ be arbitrary positive nubers and a natural $k\leq 4.$
		Given $u_0\;\in H^5(D)\cap H^2_0(D),\;D^2_x u_0(L)=0.$ 
	Then the problem (3.3)-(3.5) has a unique regular solution
	$$u\;\in L^{\infty}((0,T);H^2_0(D))\cap L^2((0,T);H^7(D));$$
	$$ u_t\; \in  L^{\infty}((0,T);L^2(D))\cap L^2((0,T);H^2_0(D)).$$

\end{thm}
\begin{proof}
	
	Define the space $W=\{f\in H^5(D)\cap H^2_0(D),\; D_x^2f(L)=0\}$ and let $\{w_i(x), \; i\in \N\}$ be a countable dense set in $W$.
	We can construct approximate solutions to (3.3)-(3.5) in the form
	$$u^N(x,t)=\sum_{i=1}^N g_i(t)w_i(x).$$
	Unknown functions $g_i(t)$\;\;satisfy the following initial problems:
	\begin{align}\frac{d}{dt}(u^N,w_i)(t)+(D^4_x u^N,w_i)(t)+(D^2_x u^N,  w_i)(t)&\notag\\
	-(D^5_x u^N,w_i)(t)+(D^3_x u^N, w_{i})(t)&\notag\\	+((u^N)^k,D_xu^N,w_i)(t)=0,&\\
		g_i(0)=g_{i0}, \;\;i=1,2,... .
	\end{align}
By Caratheodory`s existence theorem, there exist solutions of (3.6)-(3.7) at least locally in $t$, hence for all finite $N$, we can construct an approximate solution $u^N(x,t)$ of (3.6)-(3.7). In \cite{familark,larluch1}, the existence of local regular solutions to problems similar to (3.3)-(3.5) have been proved. Taking this into account,
all the estimates we will prove will be done on smooth solutions of (3.3)-(3.5). Naturally, the same estimates are true also for approximate solutions $u^N.$\\ 
\noi {\bf Estimate I }Multiply (3.3) by $2u$  to obtain 

\begin{align}\frac {d}{dt}\|u\|^2(t)+2\|D^2_{x} u\|^2(t)+2(D^2_x u,u)(t)+|D_{x}^2(0,t)|^2
=0.
\end{align}

Estimating third term in (3.8) by the Cauchy inequality, we get
\begin{align}\frac {d}{dt}\|u\|^2(t)+\|D^2_{x} u\|^2(t)+|D_{x}^2u(0,t)|^2\leq \|u\|^2(t).
\end{align}
Dropping second and third terms in (3.9) and applying Lemma 2.4, we find

\begin{align}\|u\|^2(t)\leq e^{T}\|u_0\|^2, \; t\in (0,T).
\end{align}
Returning to (3.9), we obtain

\begin{align}\int_0^T\Big[\|D^2_{x} u\|^2(t)+|D_{x}^2u(0,t)|^2\Big]dt \leq (1+Te^{T})\|u_0\|^2.
\end{align}
Here and henceforth, $T$ is an arbitrary positive number.

\noi {\bf Estimate II }

Differentiate (3.3) with respect to $t$, then multiply the result   by $2u_t$  to get
\begin{align}\frac {d}{dt}\|u_t\|^2(t)+\|D^2_{x} u_t\|^2(t)+|D_{x}^2u_t(0,t)|^2&\notag\\\leq 
\|u_t\|^2(t)	+2(u^ku_t,u_{xt})(t).
\end{align}
{\bf The case $k<4$.}\\
Making use of Lemmas 2.2 and 3.1, we estimate
\begin{align*}I=2(u^k u_t,u_{xt})(t)	\leq 2\sup_D|u(x,t)|^k\|u_t\|(t)\| u_{xt}\|(t)&\\
2\Big(2^{1/2}\|D^2_x u\|^{1/4}(t)\|u\|^{3/4}(t)\Big)^k\|u_t\|(t)\frac{1}{a^{1/2}}\| u_{xxt}\|(t)&\\
\leq\epsilon\| u_{xxt}\|^2(t)+\frac{2^k}{a\epsilon}\|D^2_x u\|^{k/2}\|u\|^{3k/2}(t)\|u_t\|^2(t)&\\
\leq\epsilon\| u_{xxt}\|^2(t)+\frac{2^k}{a\epsilon}\Big[\frac{k\|D^2_x u\|^2(t)}{4}+\frac{4-k}{4}\|u\|^{6k/(4-k)}(t)\Big]\|u_t\|^2(t).
\end{align*}
Taking $\epsilon=\frac{1}{2},$ and substituting $I$  into (3.12), we obtain

\begin{align}\frac {d}{dt}\|u_t\|^2(t)+\frac{1}{2}\|D^2_{x} u_t\|^2(t)+|D_{x}^2u_t(0,t)|^2&\notag\\
	\leq \Big(\frac{3}{2}+\frac{2^{k-1}}{a}\Big[\frac{k\|D^2_x u\|^2(t)}{4}&\notag\\+\frac{4-k}{4}(e^T\|u_0\|)^{3k/(4-k)}(t)\Big]\Big)\|u_t\|^2(t).
\end{align}
By (3.10), (3.11), $\|D^2_x u\|^2(t)\in L^1(O,T)$ and $\|u\|(t)\in L^{\infty}(0,T)$, whence,  dropping second and third terms in (3.13) and making use of Lemma 2.4, we find that

\begin{align}\|u_t\|^2(t)\leq C_1\|u_t\|^2(0), \; t\in (0,T),
\end{align}
where 
$$C_1=\exp\{\int_0^T\Big(\frac{3}{2}+\frac{2^{k-1}}{a}\Big[\frac{k\|D^2_x u\|^2(t)}{4}+\frac{4-k}{4}(e^T\|u_0\|)^{3k/(4-k)}\Big]\Big)dt\}.$$
Returning to (3.13), we obtain
\begin{align}\int_0^T\Big[\|D^2_{x} u_t\|^2(t)+|D^2_{x}u_t(0,t)|^2\Big]dt \leq C(T, \|u_0\|)\|u_t\|^2(0),
\end{align}
where $\|u_t\|(0)\leq C(\|u_0\|_W)$ can be estimated directly from (3.3) on \;$t=0.$

{\bf The case $k=4$.}\\

Consider (3.12) for $k=4$:
\begin{align}\frac {d}{dt}\|u_t\|^2(t)+\|D^2_{x} u_t\|^2(t)+|D_{x}^2u_t(0,t)|^2&\notag\\\leq \|u_t\|^2(t)
	+2(u^4u_t,u_{xt})(t).
\end{align}

Acting as by proving (3.14), we estimate
$$I_2=2(u^4u_t,u_{xt})(t)\leq 2\|u\|^4_{\infty}(t)\|u_t\|(t)\|u_{xt}\|(t)$$
$$\leq \frac{2^3}{a^{1/2}}\|D^2_x u\|(t)\|u\|^3(t)\|u_t\|(t)\|D^2_x u_t\|(t)$$
$$\leq \epsilon \|D^2_x u_t\|^2(t)+\frac{2^4}{a\epsilon}\|D^2_x u\|^2(t)\|u\|^6(t)\|u_t\|^2(t).$$ 

Taking $\epsilon=\frac{1}{2}$ and substituting $I_2$\; into (3.16), we get
\begin{align}\frac {d}{dt}\|u_t\|^2(t)+\frac{1}{2}\|D^2_{x} u_t\|^2(t)+|D_{x}^2u_t(0,t)|^2&\notag\\
	\leq \Big(\frac{3}{2}+\frac{2^5e^{3T}}{a}\Big[\|D^2_x u\|^2(t)\|u_0\|^6(t)\Big]\Big)\|u_t\|^2(t).
\end{align}
By (3.11), $\|D^2_x u\|^2(t)\in L^1(O,T)$, whence,  dropping second and third terms in (3.17) and making use of Lemma 2.4, we find that 

\begin{align}\|u_t\|^2(t)\leq C_2\|u_t\|^2(0), \; t\in (0,T),
\end{align}
where 
$$C_2=\int_0^T\Big(\frac{3}{2}+\frac{2^5e^{3T}}{a}\Big[\|D^2_x u\|^2(t)\|u_0\|^6\Big]\Big)dt.$$

Returning to (3.17), we obtain
\begin{align}\int_0^T\Big[\|D^2_{x} u_t\|^2(t)+|D^2_{x}u_t(0,t)|^2\Big]dt \leq C(T, \|u_0\|)\|u_t\|^2(0).
\end{align}
Here $\|u_t\|(0)\leq C(\|u_0\|_W)$ can be estimated directly from (3.3) on \;$t=0.$

\begin{prop} $ess\sup_{Q_T}|u(x,t)|\leq M< +\infty,\;t\leq T.$
\end{prop}
\begin{proof}
\begin{align*}
	\|D^2_x u\|^2(t)-\|D^2_x u_0\|^2=\int_0^t\frac {d}{ds}\|D^2_x u\|^2(s)ds=&\\\int_0^t2\|D^2_x\|(s)\|D^2_x u_s\|(s)ds
\leq \int_0^T\Big[\|D^2_x\|^2(s)+\|D^2_x u_s\|^2(s)\Big]ds.
\end{align*}		

Estimates 3.11), (3.19) prove that $\|D^2_x u\|(t)\in L^{\infty}(0,T)$ and Lemma 2.2 completes the proof of Proposition 3.1.

\end{proof}
Estimates (3.11), (3.18), (3.19) and Proposition 3.1 imply that $$u \in L^{\infty}((0,T);H^2_0)(D)),\;u
_t\in L^{\infty}((0,T);L^2(D))\cap L^2((0,T);H^2_0(D)).$$

These inequalities guarantee the existence of  weak soluitons to (3.3)-(3.5) $\{u(x,t)\}$ satisfying  the following integral identity:
\begin{align} (u_t,\phi)(t)+(D^2_x u,D^2_x \phi)(t)+(D^2_x u,\phi)(t)&\notag\\
+(D^2_x u,D^3_x \phi)(t)-(D^2_x u,\phi_x)(t)+(uu_x,\phi)(t)=0,\;t>0,
\end{align}
where $\phi(x,y)$ is an arbitrary function from $H^3_0(D).$\\
We can rewrite (3.20) in the form of a distribution on $(0,L)$

\begin{align}
	D^5_x u-D^4_x u-D^3_x u=u_t+D^2_x u-uu_x\equiv f(x,t).
\end{align}
Due to properties of a weak solution $u(x,t),\;\; f\in L^{\infty}(0,T);L^2(D)).$ This implies that
$$I= 	D^5_x u-D^4_x u-D^3_x u\in L^2(0,L).$$
Making use of Lemma 2.3, we find that
$$ (1-2\epsilon)\|D^5_x u\|(t)\leq \frac{C}{\epsilon}\|u\|(t) +\|f\|(t).$$
Choosing $4\epsilon=2$ and taking into account that $f,\;u\in L^2(0,L),$ we get
$$u\in  L^{\infty}(0,T);L^5(D)).$$
Taking this into account and that  $u_t\in L^2((0,T);H^2_0)),$ we find
\begin{align}
	D^5_x=u_t+D^4_xu+D^3_xu+D^2_x u+uu_x \in L^2((0,T);H^1(D)),\end{align}
whence
\begin{align*}
	u\in  L^{\infty}((0,T);H^5(D))\cap L^2((0,T);H^6(D)).
	\end{align*}
Returning to (3.22), one observes that $D^5_x u\in L^2((0,T);H^2(D)),$
hence
\begin{align}
	u\in  L^{\infty}((0,T);H^5(D))\cap L^2((0,T);H^7(D)).
\end{align}

This proves the existence part of Theorem 3.1.
\begin{rem} Assertions of Theorem 3.1 are true for arbitrary positive numbers $T,L,$ but estimates of solutions depend on $T$. It means that one can not pass to the limit as $L,\;T\to +\infty.$ Hence, we do not have  stability results. On the other hand, we have not any smallness restrictions for $u_0(x).$
\end{rem}

\begin{lemma} The regular solution of (3.3)-(3.5) is unique.
	\end{lemma}
 \begin{proof} Let $u$ and $v$ be two distinct solutions to (3.3)-(3.5). Denoting $w=u-v$, we come to the following problem:
 	
 		\begin{align}
 		w_t-D^5_xw +D^4_x w+D^3_x w +D^2_x w+\frac{1}{k+1}D_x[u^{k+1}-v^{k+1}]=0, &  \\
 		 D^i_xw(0)= D^i_xw(L)=D_x^2w(L)=0,\;i=0,1;&\\
 		w(x,0)=0.
 	\end{align}
Multiplying (3.24) by $2w$, we get 
 	
 	\begin{align} \frac{d}{dt}\|w\|^2(t)+|D^2_x w(0,t)|^2+2\|D^2_x w\|^2(t)+2(D^2_x  w,w)(t)&\notag\\
 		=\frac{2}{k+1}([u^{k+1}-v^{k+1}],D_x w)(t) \leq  \|D_x w\|^2(t)&\notag\\
 		+\frac{4}{(k+1)^2}\|u^{k+1}-v^{k+1}\|^2(t)&\notag\\
 	\leq \frac{ 1}{2} \|D^2_x w\|^2(t)+\frac{ 1}{2}\|w\|^2(t)
 	+\frac{4}{(k+1)^2}\|u^{k+1}-v^{k+1}\|^2(t.
 	 	\end{align}
Making use of the functional mean value theorem and Proposition 3.1, we get (see \cite{larluch2})

$$I=\|u^{k+1}-v^{k+1}\|^2\leq (k+1)^2\Big[\sup_{Q_T}|u|+\sup_{Q_T}|v|\Big]^{2k}\|w\|^2$$
$$\leq (k+1^2 2^{2k}M^{2k}\|w\|^2.$$

Substituting $I$ into (3.27),   we obtain
 
	\begin{align} \frac{d}{dt}\|w\|^2(t)+\frac{1}{2}\|D^2_x w\|^2(t)
	\leq C\|w\|^2(t),
\end{align}
where the constant $C$ depends on $M$.
Applying Lemma 2.4, we find that
$$ \|w\|(t)\equiv 0.$$
This proves Lemma 3.2 and consequently Theorem 3.1. 
\end{proof} 

{\bf Benney-Lin equation posed on $\R^+.$}\\
 In $Q_t=(0,t)\times R^+,\;t\in (0,T),$ consider the following problem:

\begin{align}
	u_t- D^5_x u+ D^4_x u+ D^3_x u+ D^2_x u +u^ku_x=0,&\\
	D^i_x(0,t)= 0,\;i=0,1; \;t\in(0,T);&\\
	u(x,0)=u_0(x),\;x\in R^+.
\end{align}
\begin{thm} Let $T$ be an arbitrary positive nuber and a natural $k\leq 4.$
	Given $u_0\;\in H^5(\R^+),\;u(0)=D_x(0)=0.$ 
	Then the problem (3.29)-(3.31) has a unique regular solution
	$$u\;\in L^{\infty}((0,T);H^2(\R^+))\cap L^2((0,T);H^7(\R^+));$$
	$$ u_t\; \in  L^{\infty}((0,T);L^2(\R^+))\cap L^2((0,T);H^2(\R^+)).$$
	
\end{thm}
\begin{proof}
	
	Define the space $W=\{f\in H^5(\R^+),\;f(0)=D_xu(0)=0\}$ and let $\{w_i(x), \; i\in \N\}$ be a countable dense set in $W$.
We can construct approximate solutions to (3.29)-(3.31) in the form
$$u^N(x,t)=\sum_{i=1}^N g_i(t)w_i(x).$$
Unknown functions $g_i(t)$\;\;satisfy the following initial problems:
\begin{align}\frac{d}{dt}(u^N,w_i)(t)+(D^4_x u^N,w_i)(t)+(D^2_x u^N,  w_i)(t)&\notag\\
	-(D^5_x u^N,w_i)(t)+(D^3_x u^N, w_{i})(t)&\notag\\	+((u^N)^k,D_xu^N,w_i)(t)=0,&\\
	g_i(0)=g_{i0}, \;\;i=1,2,... .
\end{align}
By Caratheodory`s existence theorem, there exist solutions of (3.32)-(3.33) at least locally in $t$, hence for all finite $N$, we can construct an approximate solution $u^N(x,t)$ of (3.32)-(3.33). In \cite{familark,larluch1}, the existence of local regular solutions to problems similar to (3.32)-(3.33) has been proved. Taking this into account,
all the estimates we will prove will be done on smooth solutions of (3.32)-(3.33). Naturally, the same estimates are true also for approximate solutions $u^N.$\\ 
\noi {\bf Estimate I }Multiplying (3.29) by $2u$ and acting in the same manner as by proving  (3.10),(3.11), we obtain 
	
\begin{align}\|u\|^2(t)\leq e^{T}\|u_0\|^2, \; t\in (0,T).
\end{align}
\begin{align}\int_0^T\Big[\|D^2_{x} u\|^2(t)+|D_{x}^2u(0,t)|^2\Big]dt \leq (1+Te^{T})\|u_0\|^2.
\end{align}
Here and henceforth, $T$ is an arbitrary positive number,\\\; $\|u\|^2(t)=\int_{\R^+}u^2(x,t)dx.$\\	
\noi {\bf Estimate II }	
Differentiate (3.29) with respect to $t$, then multiply the result   by $2u_t$  to get
\begin{align}\frac {d}{dt}\|u_t\|^2(t)+\|D^2_{x} u_t\|^2(t)+|D_{x}^2u_t(0,t)|^2&\notag\\\leq 
	\|u_t\|^2(t)	+2(u^ku_t,u_{xt})(t).
\end{align}
{\bf The case $k<4$.}\\
Making use of Lemmas 2.2, we estimate
\begin{align*}I=2(u^k u_t,u_{xt})(t)	\leq 2\sup_{\R^+}|u(x,t)|^k\|u_t\|(t)\| u_{xt}\|(t)&\\
	2\Big(2^{1/2}\|D^2_x u\|^{1/4}(t)\|u\|^{3/4}(t)\Big)^k\|u_t\|(t)\| u_{xt}\|(t)&\\
	\leq\| u_{xt}\|^2(t)+2^k\|D^2_x u\|^{k/2}\|u\|^{3k/2}(t)\|u_t\|^2(t)&\\
	\leq\epsilon\| u_{xxt}\|^2(t)
	+\frac{1}{4\epsilon}\|u_t\|^2(t)&\\+2^k\Big[\frac{k\|D^2_x u\|^2(t)}{4}+\frac{4-k}{4}\|u\|^{6k/(4-k)}(t)\Big]\|u_t\|^2(t).
\end{align*}
Taking $\epsilon=\frac{1}{2},$ and substituting $I$  into (3.36), we obtain

\begin{align}\frac {d}{dt}\|u_t\|^2(t)+\frac{1}{2}\|D^2_{x} u_t\|^2(t)+|D_{x}^2u_t(0,t)|^2&\notag\\
	\leq \Big(\frac{3}{2}+2^{k-2}\Big[k\|D^2_x u\|^2(t)&\notag\\+(4-k){4}(e^T\|u_0\|)^{3k/(4-k)}(t)\Big]\Big)\|u_t\|^2(t).
\end{align}
By (3.10), (3.11), $\|D^2_x u\|^2(t)\in L^1(O,T)$ and $\|u\|(t)\in L^{\infty}(0,T)$, whence,  dropping second and third terms in (3.37) and making use of Lemma 2.4, we find that

\begin{align}\|u_t\|^2(t)\leq C_1\|u_t\|^2(0), \; t\in (0,T),
\end{align}
where 
$$C_1=\exp\{\int_0^T\Big(\frac{3}{2}+2^{k-2}\Big[k\|D^2_x u\|^2(t)+(4-k){4}(e^T\|u_0\|)^{3k/(4-k)}(t)\Big]\Big)dt\}.$$
Returning to (3.37), we obtain
\begin{align}\int_0^T\Big[\|D^2_{x} u_t\|^2(t)+|D^2_{x}u_t(0,t)|^2\Big]dt \leq C(T, \|u_0\|)\|u_t\|^2(0),
\end{align}
where $\|u_t\|(0)\leq C(\|u_0\|_W)$ can be estimated directly from (3.29) on \;$t=0.$

{\bf The case $k=4$.}\\

Consider (3.36) for $k=4$:
\begin{align}\frac {d}{dt}\|u_t\|^2(t)+\|D^2_{x} u_t\|^2(t)+|D_{x}^2u_t(0,t)|^2&\notag\\\leq \|u_t\|^2(t)
	+2(u^4u_t,u_{xt})(t).
\end{align}

We estimate
$$I_2=2(u^4u_t,u_{xt})(t)\leq 2\|u\|^4_{\infty}(t)\|u_t\|(t)\|u_{xt}\|(t)$$
$$\leq \epsilon \|D^2_x u_t\|^2(t)+\frac{1}{4\epsilon}\|u_t\|^2(t)+2^4\|D^2_x u\|^2(t)\|u_0\|^6(t)\|u_t\|^2(t).$$ 

Taking $\epsilon=\frac{1}{2}$ and substituting $I_2$\; into (3.40), we get
\begin{align}\frac {d}{dt}\|u_t\|^2(t)+\frac{1}{2}\|D^2_{x} u_t\|^2(t)+|D_{x}^2u_t(0,t)|^2&\notag\\
	\leq \Big(\frac{3}{2}+2^4\Big[\|D^2_x u\|^2(t)\|u_0\|^6(t)\Big]\Big)\|u_t\|^2(t).
\end{align}
By (3.11), $\|D^2_x u\|^2(t)\in L^1(O,T)$, whence,  dropping second and third terms in (3.41) and making use of Lemma 2.4, we find that 

\begin{align}\|u_t\|^2(t)\leq C_3\|u_t\|^2(0), \; t\in (0,T),
\end{align}
where 
$$C_3=\int_0^T\Big(\frac{3}{2}+2^4\Big[\|D^2_x u\|^2(t)\|u_0\|^6(t)\Big]\Big)dt.$$

Returning to (3.41), we obtain
\begin{align}\int_0^T\Big[\|D^2_{x} u_t\|^2(t)+|D^2_{x}u_t(0,t)|^2\Big]dt \leq C(T, \|u_0\|)\|u_t\|^2(0).
\end{align}
Here $\|u_t\|(0)\leq C(\|u_0\|_W)$ can be estimated directly from (3.29) on \;$t=0.$
Estimates (3.34), (3.35), (3.42), (3.43) allow us to prove the existence of weak solutions to (3.29)-(3.31). Regularity and uniqueness of a regular solution can be proved in the same manner as in the proof of Theorem 3.1. 	
The proof of Theorem 3.2 is complete.

\end{proof}	
\section{ Stability intervals. Small solutions.}
Our goal in this section is to determine intervals $(0,L)$ which guarantee decay of small solutions as $t\to + \infty.$

In $Q_{t}=(0,t)\times (D),\;D=(0,L),$ consider the following initial boundary value problem:
\begin{align}
	u_t- D^5_x u+ D^4_x u+ D^3_x u+ D^2_x u +u^ku_x=0,&\\
	D^i_x(0)= D^i_x(L)=D_x^2(L)=0,\;i=0,1;&\\
	u(x,0)=u_0(x).
\end{align}

\begin{thm} Let 
	\begin{align} a=\frac{\pi^2}{L^2}>1,\;\;\theta=1-\frac{1}{a}>0,\;\text{natural}\:k\leq 8.
	\end{align}
	Given $u_{0}\;\in H^5(D)\cap H^2_0(D),\;D^2_x u_0(L)=0$ such that
	\begin{align}
	\theta-\frac{2^{k-2}}{a\theta}\Big[\frac{k\|\|u_0\|^2\|u_t\|^2(0) }{\theta^2}+(8-k)\|u_0\|^{12k/(8-k)}(t)\Big]>0
\; \text{for} \;k<8&;\notag\\
	\theta-\frac{2^8}{a\theta^3}\|u_0\|^{14}\|u_t\|^2(0)>0\; \text{for} \;k=8.
	\end{align}

	Then the problem (4.1)-(4.3) has a unique regular solution
	$$u\;\in L^{\infty}(\R^+;H^5(D)\cap H^2_0(D))\cap L^2(\R^+;H^7(D));$$
	$$ u_t\; \in  L^{\infty}(\R^+;L^2(D))\cap L^2(\R^+;H^2_0(D)).$$
	Moreover, $u$ satisfies the following inequalities:
\begin{align}\|u\|^2(t)\leq \|u_0\|^2\exp\{-2a^2\theta t\};
\end{align}
\begin{align}\|u\|^2(t)+2\theta\int_0^t\|D^2_{x} u\|^2(\tau)d\tau \leq \|u_0\|^2;
\end{align}

\begin{align}\|u_t\|^2(t)+\frac{\theta}{2}\int_0^t\|D^2_{x} u_{\tau}\|^2(\tau)d\tau \leq \|u_t\|^2(0),\;t>0;
\end{align}

\begin{align}\|u_t\|^2(t)\leq \|u_t\|^2(0)\exp\{-\frac{a^2\theta}{2} t\}.
\end{align}
	\end{thm}
\begin{proof}
To establish the results of Theorem 4.1, we use the same approach based on Faedo-Galerkin`s method as in Section 3 and we start with estimates of smooth solution of (4.1)-(4.3).

	\noi {\bf Estimate I. }Multiply (4.1) by $2u$  to obtain 
	
\begin{align}\frac {d}{dt}\|u\|^2(t)+2\|D^2_{x} u\|^2(t)-2\|D_x u\|^2(t)+|D_{x}^2(0,t)|^2
	=0.
\end{align}

Making use of Lemmas  3.1, we get

\begin{align}\frac {d}{dt}\|u\|^2(t)+2\theta\|D^2_{x} u\|^2(t)\leq 0.
\end{align}
Again, by Lemma 3.1,

\begin{align*}\frac {d}{dt}\|u\|^2(t)+2a^2\theta\| u\|^2(t)\leq 0.
\end{align*}
This implies 
\begin{align}\|u\|^2(t)\leq \|u_0\|^2\exp\{-2a^2\theta t\}.
\end{align}
Returning to (4.11), we obtain

\begin{align}\|u\|^2(t)+2\theta\int_0^t\|D^2_{x} u\|^2(\tau)d\tau \leq \|u_0\|^2.
\end{align}

\noi {\bf Estimate II.} Differentiate (4.1) by $t$ and multiply the result by $2u_t$  to obtain 

\begin{align}\frac {d}{dt}\|u_t\|^2(t)+2\theta\|D^2_{x} u_t\|^2(t)+|D_{x}^2u_t(0,t)|^2&\notag\\\leq 
	+2(u^ku_t,u_{xt})(t).
\end{align}
{\bf The case $k<8$.}\\
We estimate
\begin{align*}I=2(u^k u_t,u_{xt})(t)	\leq 2\sup_D|u(x,t)|^k\|u_t\|(t)\| u_{xt}\|(t)&\\
	2\Big(2^{1/2}\|D^2_x u\|^{1/4}(t)\|u\|^{3/4}(t)\Big)^k\|u_t\|(t)\frac{1}{a^{1/2}}\| u_{xxt}\|(t)&\\
	\leq\epsilon\| u_{xxt}\|^2(t)+\frac{2^k}{a\epsilon}\|D^2_x u\|^{k/2}\|u\|^{3k/2}(t)\|u_t\|^2(t)&\\
	\leq\epsilon\| u_{xxt}\|^2(t)+\frac{2^k}{a\epsilon}\Big[\frac{k\|D^2_x u\|^4(t)}{8}+\frac{8-k}{8}\|u\|^{12k/(8-k)}(t)\Big]\|u_t\|^2(t).
\end{align*}
Rewrite (4.11) as 
\begin{align}
	\|D^2_x u\|^2(t)\leq \frac{1}{\theta}\|u_t\|(t)\|u\|(t).
	\end{align}

Substitute  (4.15) into $I$ to get

\begin{align*}
	I=2(u^k u_t,u_{xt})(t)	\leq 2\sup_D|u(x,t)|^k\|u_t\|(t)\| u_{xt}\|(t)&\notag\\
		\leq\epsilon\| u_{xxt}\|^2(t)+\frac{2^{k-3}}{a\epsilon}\Big[\frac{k\|\|u_0\|^2\|u_t\|^2(t) }{\theta^2}&\\+(8-k)\|u_0\|^{12k/(8-k)}(t)\Big]\|u_t\|^2(t).
\end{align*}
Taking $2\epsilon=\theta$ and substituting $I$ into (4.14),
we obtain

\begin{align}\frac {d}{dt}\|u_t\|^2(t)+\frac{\theta}{2}\|D^2_{x} u_t\|^2(t)
	+\Big(\theta-\frac{2^{k-2}}{a\theta}\Big[\frac{k\|\|u_0\|^2\|u_t\|^2(t) }{\theta^2}&\notag\\+(8-k)\|u_0\|^{12k/(8-k)}(t)\Big]\Big)\|u_t\|^2(t)\leq 0.
\end{align}

Making use of (4.5), positivity of second term in (4.16) and exploiting standard arguments, we get
$$\Big(\theta-\frac{2^{k-2}}{a\theta}\Big[\frac{k\|\|u_0\|^2\|u_t\|^2(t) }{\theta^2}+(8-k)\|u_0\|^{12k/(8-k)}(t)\Big]\Big)>0,\;t>0. $$

Hence (4.16) becomes

\begin{align}\frac {d}{dt}\|u_t\|^2(t)+\frac{\theta}{2}\|D^2_{x} u_t\|^2(t)
	\leq 0.
\end{align}
Integration gives

\begin{align}\|u_t\|^2(t)+\frac{\theta}{2}\int_0^t\|D^2_{x} u_{\tau}\|^2(\tau)d\tau \leq \|u_t\|^2(0),\;t>0.
\end{align}
On the other hand, by Lemma 3.1, (4.17) can be rewritten as

\begin{align}\frac {d}{dt}\|u_t\|^2(t)+\frac{a^2\theta}{2}\|u_t\|^2(t)
	\leq 0.
\end{align}
Hence
\begin{align}\|u_t\|^2(t)\leq \|u_t\|^2(0)\exp\{-\frac{a^2\theta}{2} t\}.
\end{align}

{\bf The case $k=8$.}\\

Since (4.12), (4.13) are true for all natural $k$, consider (4.14) for $k=8$:

\begin{align}\frac {d}{dt}\|u_t\|^2(t)+2\theta\|D^2_{x} u_t\|^2(t)+|D_{x}^2u_t(0,t)|^2&\notag\\\leq 
	+2(u^8u_t,u_{xt})(t).
\end{align}
We estimate
\begin{align*}I=2(u^8 u_t,u_{xt})(t)	\leq 2\sup_D|u(x,t)|^8\|u_t\|(t)\| u_{xt}\|(t)&\\
	2\Big(2^{1/2}\|D^2_x u\|^{1/4}(t)\|u\|^{3/4}(t)\Big)^8\|u_t\|(t)\frac{1}{a^{1/2}}\| u_{xxt}\|(t)&\\
	\leq\epsilon\| u_{xxt}\|^2(t)+\frac{2^8}{a\epsilon}\|D^2_x u\|^4\|u\|^{12}(t)\|u_t\|^2(t).
\end{align*}
Taking $2\epsilon=\theta$, making use of (4.15) and substituting $I$ into (4.21), we obtain

\begin{align}\frac {d}{dt}\|u_t\|^2(t)+\frac{\theta}{2}\|D^2_{x} u_t\|^2(t)&\notag\\
	+\Big(\theta-\frac{2^{10}}{a\theta^3}\|u_0\|^{14}\|u_t\|^2(t) \Big)\|u_t\|^2(t)\leq 0.
\end{align}
Making use of (4.5), positivity of second term in (4.22) and standard arguments, we get

\begin{align}\frac {d}{dt}\|u_t\|^2(t)+\frac{\theta}{2}\|D^2_{x} u_t\|^2(t)\leq 0.
\end{align}
By Lemma 3.1, this implies

\begin{align}\|u_t\|^2(t)\leq \|u_t\|^2(0)\exp\{-\frac{a^2\theta}{2} t\};
\end{align}
\begin{align}\|u_t\|^2(t)+\frac{\theta}{2}\int_0^t\|D^2_{x} u_{\tau}\|^2(\tau)d\tau \leq \|u_t\|^2(0),\;t>0.
\end{align}
Acting in the same manner as we have proved (3.23) in the case $k\leq 4,$ we find 
\begin{align}
	u\in  L^{\infty}((0,T);H^5(D))\cap L^2((0,T);H^7(D)).
\end{align}
This proves the existence part of Theorem 4.1. Uniqueness of this regular solution has been proved in Lemma 3.2. this means that the proof of  Theorem 4.1 is complete.

\end{proof}

\end{proof}
\section{ Conclusions}

In this work,  we studied  initial boundary value problems for the generalized Benney-Lin equation (1.2) posed on bounded intervals and on a half-line. In the case of  an interval $(0,L)$, where $L$ is an arbitrary positive number, we proved the existence and uniqueness of a regular solution for $k\leq 4$ and all $t\in (0,T)$, where $T$ is an arbitrary positive number. For a special choice of $(0,L)$, we proved for $k\leq 8$ the existence and uniqueness of small regular solutions as well as their exponential decay as $t\to +\infty.$  In the case of the right half-line, we proved for $k\leq 4$ and arbitrary positive number $T$ the existence and uniqueness of a regular solution.

\section*{Conflict of Interests}

The author declares that there is no conflict of interest regarding the publication of this paper.

\section*{Data availability}
All the data that I have used are mathematical theorems  available in published articles and books, presented in References and propérly cited in the text.

\medskip

\bibliographystyle{torresmo}

\end{document}